\documentclass[11pt,reqno]{amsart}

\usepackage[T1]{fontenc}  % codifica font
\usepackage[utf8]{inputenc}  % lettere accentate
\usepackage[english]{babel} % document language

\usepackage[a4paper,hmargin=3cm,vmargin=3.5cm,heightrounded]{geometry}

\usepackage{xcolor}
\usepackage[normalem]{ulem}
\newcommand\tsout{\bgroup\markoverwith{\textcolor{red}{\rule[0.5ex]{2pt}{1.4pt}}}\ULon}
\newcommand{\stkout}[1]{\ifmmode\text{\tsout{\ensuremath{#1}}}\else\tsout{#1}\fi}

\RequirePackage[colorlinks=True,linkcolor=blue,urlcolor=blue]{hyperref} %% for coloring bibliography citations and linked URLs
\RequirePackage{graphicx} %% for including figures
\usepackage{mathrsfs} % implements \mathscr
\usepackage[font=scriptsize,labelfont={scriptsize}]{subfig} % implements \subfloat
\usepackage{bbm} % implements \mathbbm
\usepackage{enumerate}
\usepackage{enumitem,kantlipsum}
\usepackage{caption} % captions personalization
\usepackage{comment}
\usepackage{cite}
\usepackage{amssymb, amsmath}
\usepackage{nicefrac}
\usepackage{dsfont}
\usepackage{soul}

\numberwithin{equation}{section}
\setlist[enumerate]{leftmargin=*,label=(\roman*)}

\theoremstyle{plain}
\newtheorem{theorem}{Theorem}[section]

\newtheorem{lemma}[theorem]{Lemma}
\newtheorem{corollary}[theorem]{Corollary}

\theoremstyle{remark}

\newtheorem{example}[theorem]{Example}

\newtheorem{remark}[theorem]{Remark}

\newcommand{\la}{\lambda}
\newcommand{\ep}{\varepsilon}

\newcommand{\de}{\delta}
\newcommand{\si}{\sigma}
\newcommand{\ga}{\gamma}

\newcommand{\Om}{\Omega}

\newcommand{\N}{\mathbb{N}}

\newcommand{\Q}{\mathbb{Q}}
\newcommand{\R}{\mathbb{R}}

\renewcommand{\P}{\mathbb{P}}
\newcommand{\E}{\mathbb{E}}

\newcommand{\cB}{\mathcal{B}}

\newcommand{\cF}{\mathcal{F}}

\renewcommand{\d}{{\rm d}}

\newcommand{\one}{\mathbbm{1}} % mathbbm version

\newcommand{\deq}{\stackrel{\rm d}{=}}

\newcommand{\mnen}[1]{#1}

\allowdisplaybreaks
\begin{document}

\title[Aggregation of Random Variables under PRD]{Absolute Continuity of Monotone Aggregations under Positive Regression Dependence}

\author{Ben Goldys}
\address{School of Mathematics and Statistics, The University of Sydney}
\email{beniamin.goldys@sydney.edu.au}

\author{Max Nendel}
\address{Department of Statistics and Actuarial Science, University of Waterloo}
\email{mnendel@uwaterloo.ca}

\date{\today}
\thanks{\mnen{The authors thank two anonymous referees for helpful comments that led to improvements of the setup and results.}\ The first-named author is funded by the Australian Research Council, Grant No.\ DP240100781.\ The second-named author gratefully acknowledges financial support through the Deutsche Forschungsgemeinschaft (DFG, German Research Foundation) -- SFB 1283/2 2021 -- 317210226 and the Natural Sciences and Engineering Research Council of Canada via Discovery Grant no.\ RGPIN-2025-04219.\ Moreover, the second-named author would like to express his gratitude for the financial support and hospitality of the Sydney Mathematical Research Institute (SMRI)}

\begin{abstract}
   In this paper, we provide sufficient conditions for the absolute continuity of one-dimensional push-forwards of dependent random vectors.\ Suppose that $X$ has an absolutely continuous distribution and that the conditional distribution of an $\mathbb{R}^d$-valued random vector $Y$ given $X=x$ is nondecreasing in $x\in \R$ in the usual stochastic order.\ For Borel maps $g\colon \mathbb{R}\times\mathbb{R}^d\to\mathbb{R}$ satisfying a coordinatewise monotonicity condition in $Y$ and a uniform lower-increment condition in $X$, we prove that $g(X,Y)$ has an absolutely continuous distribution.\ The result requires neither independence nor a joint density, and allows the marginal law of $Y$ to be completely arbitrary.\ Moreover, the result remains valid if $\R^d$ is replaced by an arbitrary measurable space endowed with a reflexive binary relation.\ We discuss consequences for monotone risk aggregation and extensions of the familiar regularization by convolution beyond independent random variables.

    \smallskip
    \noindent \emph{Key words:}\ Absolute continuity, positive regression dependence, risk aggregation, reflexive relation.
    \smallskip
		
    \noindent \emph{MSC 2020 Classification:}\ 
    Primary 60E05; 60E15; Secondary 28A50; 49Q22.
\end{abstract}

\maketitle

\section{Introduction}

Aggregations of dependent random quantities are ubiquitous in probability, statistics, finance, insurance, and operations research, see, for instance, \cite{EcksteinKupperPohl, zbMATH06306780, zbMATH08150879}.\ A basic regularity question is whether a scalar output of the form
\[
        g(X,Y)\quad \text{with a random variable }X\in \R\text{ and a random vector }Y\in \R^d
\]
has a density when one coordinate, here $X$, already has an absolutely continuous distribution.\ If $X$ is independent of $Y$, absolute continuity of sums of the form $X+a(Y)$ with a Borel measurable function $a\colon \R^d\to \R$ follows from the fact that convolution with an absolutely continuous distribution preserves absolute continuity.\ In this case, the density of $Z=X+a(Y)$ is given by
\[
 f_Z(z)=\int_{\R^d}f_X\big(z-a(y)\big)\, \P\circ Y^{-1}(\d y)\quad \text{for all }z\in \R,
\]
where $f_X$ denotes the density of $X$ and $\P\circ Y^{-1}$ denotes the distribution of the random vector $Y$.\ 

However, in the case that the random vector $Y$ depends on the random variable $X$, and for general nondecreasing Borel measurable aggregation functions $g$, the question whether $g(X,Y)$ has a density remains open.

Clearly, negative dependence between $X$ and $Y$ rules out absolute continuity of $g(X,Y)$, in general.\ Indeed, if 
$$d=1, \quad Y=-X,\quad  \text{and} \quad g(x,y)=x+y\quad\text {for all }x,y\in \R,$$ we get $g(X,Y)\equiv 0$.\ Therefore, under negative dependence, absolute continuity cannot be expected.\ 

The aim of this note is to identify sufficient conditions for the absolute continuity of the aggregate $g(X,Y)$, where $Y$ takes values in an arbitrary measurable space equipped with a reflexive binary relation.

Intuitively, positive dependence between $X$ and $Y$ should prevent probability mass from concentrating on the level sets of the map $(x,y)\mapsto x+y$, and hence rule out the formation of atoms in the distribution of $X+Y$ if $X$ already has a density.\ However, even if this intuition is correct, it would only mean continuity of the distribution function of $X+Y$, whereas absolute continuity is a stronger property.\ It is therefore far from clear whether absolute continuity of $g(X,Y)$ can be guaranteed when $X$ has a density and $Y$ depends positively on $X$.\ This difficulty is compounded by the fact that there is no single unified notion of positive dependence:\ there are many non-equivalent concepts of positive dependence, and it is not a priori clear which one is appropriate for establishing absolute continuity of the aggregate.

The main result of this paper, Theorem \ref{thm.main}, shows that absolute continuity of $g(X,Y)$ can indeed be guaranteed if $Y$ is {positively regression dependent} (PRD) on $X$, that is, the conditional law of $Y$ given $X=x$ is nondecreasing in the usual stochastic order $\leq_{\rm st}$ as a function of \mnen{$x\in I$, where $I$ is an interval with $X\in I$ $\P$-a.s., cf.\ \cite[Equation (5.3)]{zbMATH03236593} and \cite{zbMATH03141657},} and $g$ is sufficiently increasing in the first variable.\ A key step of the proof is the reduction to the case where $Y$ is real-valued and $g(X,Y)=X+Y$.\ The intuition described above, namely that positive dependence prevents the creation of atoms is then made precise by explicitly proving the absolute continuity of the distribution function of $X+Y$.\ 

We refer to \cite{zbMATH05043275} for a similar phenomenon arising in the context of L\'evy processes with drift.\ There, the authors show that, for a function $a\colon \R\to \R$ with bounded continuous derivative and a L\'evy process $(Y_t)_{t\geq 0}$ with infinite jump measure, the random variable
\[
X_1=x+\int_0^1 a(X_s)\, \d s+Y_1
\]
has a density w.r.t.\ the Lebesgue measure if $a$ is strictly increasing in a neighborhood of the initial value $x\in \R$. 

We point out that positive regression dependence is a well-established concept of positive dependence.\ In the bivariate case, it is well-known that positive regression dependence of $Y$ on $X$ implies positive quadrant dependence (PQD), cf.\ \cite[Lemma 4]{zbMATH03236593}, which, in this case, coincides with positive lower orthant dependence (PLOD), cf.\ \cite[Chapter 6]{Denuit1}.\ Moreover, for bivariate random vectors $(X,Y)$, the random variables $X$ and $Y$ are {conditionally increasing} (CI), i.e., positively regression dependent on each other in the usual distributional sense, if and only if \mnen{there exists a} copula of the random vector $(X,Y)$ \mnen{that} is concave in each variable while keeping the other variable fixed, cf. \cite[Theorem 3.7]{zbMATH02228221} \mnen{together with 
\cite[Theorem~3]{zbMATH08046505}}.\ We also refer to \cite{zbMATH02167809} for further characterizations of CI for Archimedean copulae and to \cite{zbMATH08027860} for a discussion of positive dependence in the context of insurance pricing.\ Observe that comonotonicity, which usually serves as an extreme benchmark for positive dependence, see, e.g., \cite{Embrechts2002} and \cite[Definition 2.13]{zbMATH06090392}, is a particular form of PRD.\ However, comonotonic couplings are typically singular, so that absolute continuity of the aggregate cannot simply be drawn as a consequence of joint absolute continuity.\ Despite the potential singularity of the joint distribution, Theorem \ref{thm.main} still ensures the existence of a density for the aggregate random variable $g(X,Y)$.

The condition that $g$ is sufficiently increasing in $X$ cannot be simply omitted either.\ If $g$ is, for example, constant in the first variable, absolute continuity of $g(X,Y)$ cannot be expected.\ For example, if $Y\equiv0$ and $g$ is constant in the first variable, it follows that $g(X,Y)\equiv g(0,0)$ does not even have a  continuous cumulative distribution function.\

Another subtlety lies in the fact that there are continuous distribution functions that are not absolutely continuous, the most prominent example being the Cantor function or devil's staircase.\ Already in the case where $g$ is independent of $Y$, strict monotonicity of $g$ in the first variable is not enough to ensure that $g(X,Y)$ is absolutely continuous.\ If, for example, $g$ is independent of the second variable and given as the inverse of the function $\frac12\big(x+c(x)\big)$ on the interval $(0,1)$, where $c$ is the Cantor function, it follows that $g(X)$ is \textit{not} absolutely continuous even if $X$ is, for example, uniform on $(0,1)$, see Example \ref{ex.counterexample} for the details.\ In order to exclude these cases, one therefore has to assume that $g\colon \mnen{I}\times \R^d\to \R$ is Borel measurable with
\begin{equation}\label{eq.aggregation.property.intro}
g(x_2,y_2)-g(x_1,y_1)\geq h(x_2)-h(x_1)
\end{equation}
for all $x_1,x_2\in \mnen{I}$ with $x_1\leq x_2$ and  $y_1,y_2\in \R^d$ with $y_1\leq y_2$, where $h\colon \mnen{I}\to \R$ is strictly increasing with locally absolutely continuous inverse, see Section \ref{sec:main} for the details.\ Observe that Condition \eqref{eq.aggregation.property.intro} implies that $g$ is nondecreasing in both variables.

\mnen{Several counterexamples highlight that natural weakenings of our conditions fail.\ Using the previously mentioned function $\frac{1}2\big(x+c(x)\big)$, Example \ref{ex.counterexample} shows that the assumptions on $h$ cannot be weakened to requiring that $h$ is nondecreasing, bijective, and uniformly continuous.\ Another example shows that the assumption that $Y$ is positively regression dependent on $X$ cannot be replaced by the slightly more general notion of left-tail decreasingness (LTD), which is still stronger than PQD, see Example \ref{ex.strictpqd} for the details.}

\mnen{We point out that a standard sufficient condition to ensure absolute continuity of the aggregate $g(X,Y)$ requires that the conditional law of $X$ given $Y=y$ is absolutely continuous w.r.t.\ the Lebesgue measure and $I\to \R,\, x\mapsto g(x,y)$ is strictly monotone and continuous with locally absolutely continuous inverse for $\P\circ Y^{-1}$-almost every $y$, see Remark \ref{rem.alternative} for the details.\ In contrast, our result replaces absolute continuity of the conditional law of $X$ given $Y$ by the existence of a density for the marginal law of $X$ and stochastic monotonicity of the conditional law of $Y$ given $X$.\ This allows singular conditional and joint laws and opens for a different route to verify absolute continuity for the law of the aggregate.}

In the following section, we state and prove the main result, Theorem \ref{thm.main}, for a general measurable space, endowed with a reflexive binary relation, instead of $\R^d$ and draw several consequences from it.

\section{Main result}\label{sec:main}
Let $(\Om,\cF,\P)$ be a probability space and $(S,\mathscr S)$ be a measurable space, endowed with an arbitrary reflexive binary relation $\precsim$.\ We refer to \cite[Section I.1]{zbMATH03796014} and \cite[Chapter 1]{zbMATH01748069} for a survey on the less general, yet in most applications standard, concept of a partial order, i.e., a reflexive, transitive, and antisymmetric relation.\ Moreover, let $X\colon \Om\to \R$ be a random variable and $Y\colon \Om\to S$ be an $\cF$-$\mathscr S$-measurable map.\ \mnen{Finally, let $I\subseteq \R$ be an interval, i.e., $[a,b]\subseteq I$ for all $a,b\in I$ with $a\leq b$, with $X\in I$ $\P$-a.s.\ We write $\cB$ for the Borel $\si$-algebra on $\R$ and endow $I$ with the trace $\si$-algebra $\cB_I$ of $\cB$.}

For a random variable $Z\colon \Om\to \R$, we denote by $F_Z\colon \R\to [0,1]$ the distribution function of $Z$, i.e.,
\[
F_Z(a):=\P(Z\leq a)\quad\text{for all }a\in \R.
\]
A map $g\colon \mnen{I}\times S\to \R$ is called \textit{nondecreasing} if $$g(x_1,y_1)\leq g(x_2,y_2)\quad \text{for all }x_1,x_2\in \mnen{I}\text{ with }x_1\leq x_2\text{ and }y_1,y_2\in S\text{ with }y_1\precsim y_2.$$ For functions that are independent of the first variable, this also yields a notion of a nondecreasing map $f\colon S\to \R$.\ Since $\precsim$ is reflexive, a function $f\colon \mnen{I}\to \R$ is nondecreasing in the sense of the above definition (viewed as a function that is independent of the second variable) if and only if it is nondecreasing in the usual sense.

We say that $Y$ is \textit{positively regression dependent} (PRD) on $X$ if there exists a stochastic kernel $\pi\colon \mnen{I}\times  \mathscr S\to [0,1]$ such that
\begin{equation}\label{eq.kernelrep}
 \E\big[g(X,Y)\big]=\int_{\R} \int_S g(x,y)\,\pi(x,\d y)\,F_X(\d x)
\end{equation}
 for all bounded $\mnen{\cB_I}\otimes \mathscr S$-measurable maps $g\colon \mnen{I}\times S\to \R$ and 
\begin{equation}\label{eq.def.prd}
\pi(x_1,\,\cdot\,)\leq_{\rm st} \pi(x_2,\,\cdot\,) \quad \text{for all }x_1,x_2\in \mnen{I}\text{ with }x_1\leq x_2,
\end{equation}
 i.e., for all $x_1,x_2\in \mnen{I}$ with $x_1\leq x_2$ and 
 every nondecreasing bounded $\mathscr S$-measurable function $f\colon S\to \R$,
 \[
 \int_{S} f(y)\,\pi(x_1,\d y)\leq \int_{S} f(y)\,\pi(x_2,\d y).
 \]
In other words, $Y$ is PRD on $X$ if the conditional law of $Y$ given $X=x$ is nondecreasing in $x\in \mnen{I}$ with respect to the stochastic order induced by the reflexive relation $\precsim$ on $S$.\ In the case $S=\R^d$ with the coordinatewise order, the stochastic order $\leq_{\rm st}$ corresponds to the \textit{usual stochastic order}.\ We refer to \cite{zbMATH03236593} for a broad discussion of the concept of positive regression dependence and to \cite[Chapter 6]{zbMATH05076652} for an overview on multivariate stochastic orders such as the usual stochastic order $\leq_{\rm st}$ in the case $S=\R^d$.\

We point out that the existence of the stochastic kernel $\pi$ is part of the definition of positive regression dependence.\ If $S$ is a Borel space, the existence of a stochastic kernel, \mnen{defined on $\R$ and} satisfying \eqref{eq.kernelrep} for all \mnen{bounded} $\cB\otimes \mathscr S$-measurable maps $g\colon \R\times S\to \R$, follows from \cite[Theorem 3.4 (i)]{zbMATH07301635}.\ In this case, for every $B\in \mathscr S$, the map $\R\to [0,1],\, x\mapsto \pi(x,B)$ is $\P\circ X^{-1}$-a.s.\ uniquely specified, cf.\ \cite[Theorem 3.4 (ii)]{zbMATH07301635}\mnen{, and $\pi$ is often called the \textit{regular conditional distribution} of $Y$ given $X$}.

\mnen{We say that a strictly increasing function $h\colon I\to \R$ has a \textit{locally absolutely continuous inverse} if the inverse $h^{-1}\colon \overline{h}(I)\to \R$ is locally absolutely continuous, i.e., $h^{-1}|_{[a,b]}$ is absolutely continuous for all $a,b\in \overline h(I)$ with $a\leq b$.\ Here and throughout, 
\begin{equation}\label{eq.hbar} 
\overline h(I):=\bigcup_{\substack{a,b\in h(I)\\ a\leq b}} [a,b]
\end{equation}
is the convex hull of $h(I)$, which is an interval, and
\begin{equation}\label{eq.inverseh}
h^{-1}(x):=\sup_{\substack{a\in h(I)\\ a\leq x}} h^{-1}(a)=\inf_{\substack{b\in h(I)\\ x\leq b}} h^{-1}(b)\in I\quad \text{for all }x\in \overline h(I)\setminus h(I)
\end{equation}
extends the inverse $h^{-1}$, defined on the image $h(I)$, constantly on holes of $h(I)$ by the common one-sided limiting values of the inverse at the endpoints of each hole.}

The following theorem is the main result of this paper.

\begin{theorem}\label{thm.main}
Let $g\colon \mnen{I}\times S\to \R$ be $\mnen{\cB_I}\otimes \mathscr S$-measurable with
\begin{equation}\label{eq.aggregation.property}
g(x_2,y_2)-g(x_1,y_1)\geq h(x_2)-h(x_1)
\end{equation}
for all $x_1,x_2\in \mnen{I}$ with $x_1\leq x_2$ and  $y_1,y_2\in S$ with $y_1\precsim y_2$, where $h\colon \mnen{I}\to \R$ is \mnen{strictly increasing} with \mnen{locally} absolutely continuous inverse.\ If $F_X$ is absolutely continuous and $Y$ is PRD on $X$, then the distribution function $F_{g(X,Y)}$ of $g(X,Y)$ is \mnen{locally} absolutely continuous.\ In particular, $g(X,Y)$ admits a density w.r.t.\ the Lebesgue measure on $\cB$.
\end{theorem}

\mnen{Before we turn our focus on the proof of Theorem \ref{thm.main}, we briefly discuss its assumptions.}

Note that \eqref{eq.aggregation.property} is equivalent to
\[
 g(x,y_1)\leq g(x,y_2)\quad\text{for all }x\in \mnen{I}\text{ and }y_1,y_2\in S\text{ with }y_1\precsim y_2
\]
and
\[
g(x_2,y)-g(x_1,y)\geq h(x_2)-h(x_1)\quad\text{for all }x_1,x_2\in \mnen{I}\text{ with }x_1\leq x_2\text{ and }y\in S.
\]
In particular, \eqref{eq.aggregation.property}  implies that $g$ is \textit{nondecreasing}.

\mnen{Observe that the class of functions $h$ that can be subsumed under our framework also covers functions, where the inverse fails to be uniformly continuous, such as $h(x)=\sqrt{x}$ for $x\in I$ with $I=(0,\infty)$, and functions that fail to be surjective, such as $h(x)=\arctan(x)$ for $x\in \R$.\ However, already in the case where $g$ does not depend on the second variable, the hypothesis that $h$ has a locally absolutely continuous inverse cannot be weakened  in Theorem \ref{thm.main}, not even to requiring that $h$ is nondecreasing and  bijective with a uniformly continuous inverse as the following example shows.} 

\begin{example}\label{ex.counterexample}
\mnen{Let} $g(x,y):=h(x)$ for $x\in \R$, where $h$ is the inverse of $\frac{1}2\big(x+c(x)\big)$ for $x\in (0,1)$ with $c$ being the Cantor function and $h(x):=x$ for $x\in \R\setminus(0,1)$.\ Since the Cantor function is uniformly continuous, it follows that $$h^{-1}(x)=\begin{cases}\frac{1}2\big(x+c(x)\big),&x\in (0,1),\\ x, &x\in \R\setminus(0,1),\end{cases}$$ is uniformly continuous.\ Then, choosing $X$ to be uniform on $(0,1)$, it follows that $F_{h(X)}(a)=\frac{1}2\big(a+c(a)\big)$ for $a\in (0,1)$, which is not absolutely continuous, since the Cantor function fails to be absolutely continuous.\
\end{example}

\mnen{We now turn our focus on the assumption that $Y$ is PRD on $X$.\ For the sake of a simplified exposition, we only consider the case $S=\R$.\ A well-known weaker concept of positive dependence than PRD is that of left-tail decreasingness, cf.\ \cite[Equation (5.2)]{zbMATH03236593}.\ The random variable $Y$ is called \textit{left-tail decreasing} (LTD) in $X$ if
\begin{equation}\label{eq.def.ltd}
 \P\big(\{X\leq a_1\}\cap \{Y\leq b\}\big)\P(X\leq a_2)\geq \P\big(\{X\leq a_2\}\cap \{Y\leq b\}\big)\P(X\leq a_1)
\end{equation}
for all $a_1,a_2\in I$ with $a_1\leq a_2$ and $b\in \R$.\ A straightforward estimate shows that $Y$ is LTD in $X$ if $Y$ is PRD on $X$, see, for instance, \cite[Lemma 4]{zbMATH03236593} for the case $I=\R$.}

\mnen{Another prominent form of positive dependence is positive quadrant dependence.\ A random vector $(X,Y)$ is said to be \textit{positively quadrant dependent} (PQD) if
\[ 
\P\big(\{X\leq a\}\cap \{Y\leq b\}\big)\geq \P(X\leq a)\P(Y\leq b)\quad \text{for all }a\in I\text{ and }b\in \R,
\]
see, for instance, \cite[Equation (2.1)]{zbMATH03236593}.\ 
We point out that, for $I=\R$, PQD is a symmetric dependence relation in the sense that the random vector $(X,Y)$ is PQD if and only if $(Y,X)$ is PQD.\ However, choosing $a=a_1$ and $a_2=\sup I$ or letting $a_2\uparrow \sup I$ in \eqref{eq.def.ltd}, the one-sided condition of $Y$ being LTD in $X$ implies that the random vector $(X,Y)$ is PQD, see also \cite[Lemma 4]{zbMATH03236593}.\ The following example shows that the assertion of Theorem \ref{thm.main} is false if PRD is replaced by the slightly weaker condition of LTD, which is still stronger than PQD.}

\begin{example}\label{ex.strictpqd}
\mnen{Let $I:=(0,1)$, $X$ be uniform on $(0,1)$, $B$ be independent of $X$ with $P(B=1)=\frac12$, and
\[
 Y:=\begin{cases}
  1-X,&\text{if }B=1\text{ and }X\in \big(\tfrac12,\tfrac23\big],\\
  X,&\text{otherwise}.
 \end{cases}
\]
By construction, $0<Y\leq X<1$ $\P$-a.s.\ and
\begin{align*}
 F_{(X,Y)}(u,v)&:=\P\big(\{X\leq u\}\cap \{Y\leq v\}\big)\\
 &\;=\P\big(\{X\leq u\}\cap \{X\leq v\}\big)+\P\big(\{v<X\leq u\}\cap \{Y\leq v\}\big)\\
 &\;=\min\{u,v\}+\frac12\P\big(\{v<X\leq u\}\cap \big\{\tfrac12<X\leq \tfrac23\big\} \cap \{1-X\leq v\}\big)\\
 &\;=\min\{u,v\}+\frac12\P\Big(\Big\{\max\{v,1-v\}<X\leq \min\big\{u,\tfrac23\big\}\Big\}\Big)
\end{align*}
for all $u,v\in (0,1)$.\ Since $X$ is uniform on $(0,1)$, \eqref{eq.def.ltd} is satisfied if and only if the map
\[
(0,1)\to \R,\quad u\mapsto \frac{F_{(X,Y)}(u,v)}{u}
\]
is nonincreasing for all $v\in (0,1)$.\ If $u>\tfrac23$ or $u\leq \max\{1-v,v\}$ this is obvious.\ If $1-v\leq \tfrac12<u\leq \tfrac23$, it follows that
\[
 \frac{F_{(X,Y)}(u,v)}u=\min\big\{1,\tfrac{v}u\big\}+\frac12\max\big\{0,1-\tfrac{v}u\big\}=\max\big\{1,\tfrac12\big(1+\tfrac{v}{u}\big)\big\}
\]
is nonincreasing in $u$, and it remains to consider the case where $\tfrac12<1-v<u\leq \tfrac23$.\ Then,
\[
 \frac{F_{(X,Y)}(u,v)}u=\frac12+\frac{3v-1}{2u},
\]
and $\tfrac{3v-1}2>0$ since $1-v<\tfrac23$, so that $\frac{F_{(X,Y)}(u,v)}u$ is also nonincreasing.\ Since $F_{(X,Y)}$ is continuous, we have therefore verified that \eqref{eq.def.ltd} holds.\ Now, for $g(x,y):=x+y$, we have
\[
 \P(X+Y=1)=\P\big(\{B=1\}\cap \big\{\tfrac12< X\leq\tfrac23\}\big)=\frac1{12},
\]
so that $X+Y$ is not even atomless.\ Observe that $Y$ is not PRD on $X$ since every stochastic kernel $\pi\colon I\times \mathcal B_I\to [0,1]$ with \eqref{eq.kernelrep} satisfies
\[
\pi\big(x, \big(\tfrac25,1\big)\big)=1 \quad \text{for Lebesgue a.e.\ }x\in \big(\tfrac12,\tfrac35\big)
\]
and
\[
\pi\big(x, \big(\tfrac25,1\big)\big)=\frac12 \quad \text{for Lebesgue a.e.\ }x\in \big[\tfrac35,\tfrac23\big].
\]}
\end{example}

\mnen{Moreover, we point out that we only require the kernel $\pi$ to be stochastically nondecreasing on $I$, not on $\R$. Clearly, the kernel $\pi$ can be extended from $I$ to a kernel on $\R$ retaining \eqref{eq.kernelrep}, e.g., by assigning the same (arbitrary) probability measure to all elements of the $\P\circ X^{-1}$-null set $\R\setminus I$.\ However, a stochastically nondecreasing kernel on an interval $I$ with $X\in I$ $\P$-a.s., in general, neither admits a stochastically nondecreasing extension to $\R$ nor the closure of $I$, not even after modification on a $\P\circ X^{-1}$-null set, as the following example shows.}

\begin{example}
Let $X$ be exponential with intensity $1$, $I=(0,\infty)$, and $Y:=\log X$.\ Then, the regular conditional distribution
\[
  \pi(x,\,\cdot\,)=\delta_{\log x},
  \quad \text{for }x\in(0,\infty),
\]
is stochastically nondecreasing on $(0,\infty)$, where $\delta_y$ denotes the Dirac measure with barycenter $y\in \R$, so that $Y$ is PRD on $X$.

Now, let $\mu$ be a  Borel probability measure on $\R$ with $\mu\leq_{\rm st} \pi(x,\, \cdot\,)=\delta_{\log x}$ for $\P\circ X^{-1}$-almost every $x\in (0,\infty)$, i.e., Lebesgue-almost every $x\in (0,\infty)$.\ Then, for each $n\in \N$, we can choose $x_n\in (0,e^{-n}]$ with $\mu\le_{\rm st}\delta_{\log x_n}$, so that
\[
  \mu\big((-n,\infty)\big)
  \le \delta_{\log x_n}\big((-n,\infty)\big)=0,
\]
which contradicts
\[
  1=\mu(\R)
   =\lim_{n\to\infty}\mu\big((-n,\infty)\big)=0.
\]
Thus, no stochastically nondecreasing extension of $\pi$ to the closure $\overline I=[0,\infty)$ of $I$ exists.
\end{example}

\begin{remark}\label{rem.alternative}
\mnen{A classical but slightly different route to absolute continuity of the aggregate $g(X,Y)$ is to condition
on $Y$ and require absolute continuity of the conditional probabilities $\P\circ Y^{-1}$-a.s.\ To be more precise, let $(S,\mathscr S)$ be an arbitrary measurable space and $\gamma\colon S\times \cB_I\to [0,1]$ be a stochastic kernel with
\begin{equation}\label{eq.disint}
\E[g(X,Y)]=\int_S\int_I g(x,y)\gamma(y,\d x)\, \P\circ Y^{-1}(\d y)
\end{equation}
for all bounded $\cB_I\otimes \mathscr S$-measurable functions $g\colon I\times S\to \R$.\
The existence of such a kernel that is $\P\circ Y^{-1}$-a.s.\ unique is guaranteed by the fact that $I$ is a Borel space, cf.\ \cite[Theorem 3.4]{zbMATH07301635}.\ Moreover, let $g\colon I\times S\to \R$ be $\cB_I\otimes \mathscr S$-measurable.\ For $\P\circ Y^{-1}$-a.e.\ $y\in S$, we assume that $g_y:=g(\,\cdot\, ,y)$ is strictly increasing and continuous with locally absolutely continuous inverse and $\gamma (y,\,\cdot\,)$ is absolutely continuous w.r.t.\ the Lebesgue measure.\ Then, the distribution of $g(X,Y)$ is absolutely continuous w.r.t.\ the Lebesgue measure.\ Indeed, let $N\in \cB$ be a Lebesgue null.\ Since $g_y^{-1}$ is locally absolutely continuous, by \cite[Exercise 5.8.49]{bogachev-2007}, the set $\{x\in I\colon g(x,y)\in N\}$ is a Lebesgue null for  $\P\circ Y^{-1}$-a.e.\ $y\in S$. Since $\gamma(y,\, \cdot\,)$ is absolutely continuous w.r.t.\ the Lebesgue measure for $\P\circ Y^{-1}$-a.e.\ $y\in S$, it follows that
\[
\int_I\one_N\big(g(x,y)\big)\,\gamma(y,\d x)=0
\quad\text{for }\P\circ Y^{-1}\text{-a.e.\ }y\in S.
\]
Equation \eqref{eq.disint} now gives
\[
\P\big(g(X,Y)\in N\big)
=
\int_S\int_I
\one_N\big(g(x,y)\big)\,\gamma(y,\d x)\,\P\circ Y^{-1}(\d y)=0.
\]}
\end{remark}

We prepare the proof of Theorem \ref{thm.main} with the following auxiliary result, cf.\ \cite[Corollary 5.4.4]{bogachev-2007}. For the sake of a self-contained exposition, we provide a short proof.

\begin{lemma}\label{lem.preparation}
 Let $F_X$ be \mnen{locally} absolutely continuous and $h\colon \mnen{I}\to \R$ be \mnen{strictly increasing} with \mnen{locally} absolutely continuous inverse.\ Then, $F_{h(X)}$ is \mnen{locally} absolutely continuous.
\end{lemma}

\begin{proof}
For $a,b\in \R$ with $a<b$, let
\[
I_{a,b}:=\big\{x\in \mnen{I}\colon h(x)\in (a,b]\big\}.
\] 
Since, for $a_1,b_1,a_2,b_2\in \R$ with $a_1<b_1\leq a_2< b_2$, the intervals $I_{a_1,b_1}$ and $I_{a_2,b_2}$ are disjoint due to the strict monotonicity of $h$, it follows that \mnen{$F_{h(X)}|_{\overline h(I)}=F_X\circ h^{-1}$ is locally} absolutely continuous.\ Indeed, \mnen{let $a,b\in \overline h(I)$ with $a\leq b$. Since $h^{-1}\colon \overline h(I)\to \R$ is nondecreasing, there exist $c,d\in \R$ with $c\leq d$ such that the image of $[a,b]$ under $h^{-1}$ is contained in $ [c,d]$, i.e., $(h^{-1})([a,b])\subseteq [c,d]$.\ S}ince $F_X$ is \mnen{locally} absolutely continuous, for all $\ep>0$, there exists some $\ga>0$ such that $$\sum_{k=1}^n{F_X(d_k)-F_X(c_k)}<\ep\quad \text{if} \quad\sum_{k=1}^n{d_k-c_k}<\ga$$ for all $n\in \N$ and $c_k,d_k\in \R$ with \mnen{$c\leq c_k<d_k\leq c_{k+1}$} for all $k=1,\ldots, n$ and \mnen{$c_{n+1}:=d$}.\ On the other hand, since $h^{-1}$ is locally absolutely continuous, there exists some $\de>0$ such that $$\sum_{k=1}^n{h^{-1}(b_k)-h^{-1}(a_k)}<\ga\quad\text{if} \quad\sum_{k=1}^n{b_k-a_k}<\de$$ for all $n\in \N$ and $a_k,b_k\in \R$ with \mnen{$a\leq a_k<b_k\leq a_{k+1}$} for all $k=1,\ldots, n$ and \mnen{$a_{n+1}:=b$}.\ Choosing $c_k:=h^{-1}(a_k)$ and $d_k:=h^{-1}(b_k)$ for such $a_k,b_k\in \R$, the local absolute continuity of $F_{h(X)}|_{\overline h(I)}$ follows.\ This implies that $F_{h(X)}$ is locally absolutely continuous since $X\in I$ $\P$-a.s.\ and $X$ has no atoms.
\end{proof}

We are now ready to prove the main result, Theorem \ref{thm.main}. 

\begin{proof}[Proof of Theorem \ref{thm.main}]
By Lemma \ref{lem.preparation}, we may w.l.o.g.\ assume that \mnen{$h={\rm id}|_{I}$} if we consider \mnen{$\overline X:=h(X)$} instead of $X$, \mnen{$\overline h(I)$ instead of $I$, $x+g\big(h^{-1}(x), \,\cdot\,\big)-h\big(h^{-1}(x)\big)$ for $x\in \overline h(I)$} instead of $g$, and 
$\pi\big(h^{-1}(\,\cdot\,),\,\cdot\, \big)$ instead of $\pi$.

 By \eqref{eq.aggregation.property}, the map $\overline g\colon \mnen{I}\times S\to \R,\, (x,y)\mapsto g(x,y)- x$ is nondecreasing, and it follows that $$\mnen{\overline Y:=\overline g(X,Y)= g(X,Y)- X} \text{ is PRD on } X.$$ Indeed, let $\pi\colon \mnen{I}\times  \mathscr S\to [0,1]$ be a stochastic kernel such that
\eqref{eq.kernelrep} and \eqref{eq.def.prd} hold.\ Then, for every nondecreasing bounded function $f\colon \R\to \R$,
\[
 \int_{S} f\big(\overline g(x_1,y)\big)\,\pi(x_1,\d y)\leq \int_{S} f\big(\overline g(x_1,y)\big)\,\pi(x_2,\d y)\leq \int_{S} f\big(\overline g(x_2,y)\big)\,\pi(x_2,\d y)
\]
\mnen{for all $x_1,x_2\in I$ with $x_1\leq x_2$,} where, in the second step, we used the reflexivity of the relation $\precsim$.\
Hence, choosing the stochastic kernel
\[
 \overline\pi(x,B):=\int_{S} \one_B\big(\overline g(x,y)\big)\,\pi(x,\d y)\quad \text{for all }x\in \mnen{I}\text{ and }B\in \cB,
\] one finds that \mnen{$\overline Y$} is PRD on $X$ and \mnen{$g(X,Y)=X+\overline Y$}.\

We may therefore w.l.o.g.\ assume that $S=\R$, $\mathscr S=\cB$, and $g(x,y)=x+y$ for all $x,y\in \R$ with \mnen{$Y=\overline Y$}.\ Since $Y$ is PRD on $X$, the \mnen{left-continuous conditional quantile
$$
f(x,u):= \min\big\{a\in \R\colon \pi\big(x,(-\infty,a]\big)\geq u\big\}\quad\text{for all }x\in \mnen{I}\text{ and }u\in (0,1)
$$ 
is nondecreasing.\ Observe that the minimum is attained since the map $\R\to [0,1],\, a\mapsto \pi\big(x,(-\infty,a]\big)$ is right-continuous for all $x\in I$.\ Moreover, the function $f$ is Borel measurable since $$\big\{(x,u)\in \mnen{I}\times (0,1)\colon f(x,u)\leq b\big\}=\bigcap_{\substack{a\in \Q\\ a>b}}\Big\{(x,u)\in \mnen{I}\times (0,1)\colon \pi\big(x,(-\infty,a]\big)\geq u\Big\}$$ for all $b\in \R$.}\ By potentially passing to a different probability space, we may w.l.o.g.\ assume that there exists a random variable $U$ that is uniform on $(0,1)$ and independent of $X$, so that
\begin{equation}\label{eq.deq}
(X,Y) \deq \big(X,f(X,U)\big).
\end{equation}
Observe that, for fixed $u\in (0,1)$ and $a,b\in \R$ with $a<b$, the set
\[
I_{a,b}(u):=\big\{x\in \mnen{I}\colon x+f(x,u)\in (a,b]\big\}
\]
is an interval (possibly empty or a singleton) of length at most $b-a$, since $f(\,\cdot\,,u)$ is nondecreasing.\ Indeed, let $u\in (0,1)$, $a,b\in \R$ with $a<b$, and $x_1,x_2\in I_{a,b}(u)$ with $x_1\leq x_2$.\ Then,
\begin{align*}
x_2-x_1&\leq x_2-x_1+f(x_2,u)-f(x_1,u)= \big(x_2 +f(x_2,u)\big) - \big(x_1+f(x_1,u)\big)\leq b-a.
\end{align*}
Now, let \mnen{$\ep\in (0,1)$} and $\la$ denote the Lebesgue measure on $\cB$, \mnen{and $a,b\in \R$ with $a\leq b$.\ Then, there exist $c,d\in \R$ with $c\leq d$ and $I_{a,b}(u)\subseteq [c,d]$ for all $u\in \big[\frac\ep3,1-\frac\ep3\big]$.\ Indeed, since $f$ is nondecreasing, for all $x\in \R$ and $u\in \big[\frac\ep3,1-\frac\ep3\big]$,
\[
 I_{a,b}(u)\subseteq \Big\{x\in I\colon a< x+f\big(x,1-\tfrac\ep3\big)\text{ and } x+f\big(x,\tfrac\ep3\big)\leq b\Big\},
\]
and the set on the right-hand side is bounded.} Since $F_X$ is \mnen{locally} absolutely continuous, there exists some $\de>0$ such that
\[
\int_{\mnen{[c,d]}} \one_B(x)\, F_X(\d x)<\frac\ep3\quad \text{for all } B\in \cB\text{ with }\la(B)<\de.
\]
Let $n\in \N$, \mnen{$a_{n+1}:=b$}, and $a_k,b_k\in \R$ with \mnen{$a\leq a_k< b_k\leq a_{k+1}$} for all $k=1,\ldots, n$ and $\sum_{k=1}^n{b_k-a_k}<\de$.\ Then, by the previous step and since preimages of pairwise disjoint sets are pairwise disjoint,
$$
\lambda \bigg(\bigcup_{k=1}^n I_{a_k,b_k}(u)\bigg)= \sum_{k=1}^n \la\big(I_{a_k,b_k}(u)\big)\leq \sum_{k=1}^n b_k-a_k<\de.
$$
Hence, by \eqref{eq.deq} and Tonelli's theorem,
\begin{align*}
\sum_{k=1}^n F_{X+Y}(b_k)-F_{X+Y}(a_k)&=\sum_{k=1}^n F_{X+f(X,U)}(b_k)-F_{X+f(X,U)}(a_k)\\
&=\sum_{k=1}^n\int_0^1 \int_\R  \one_{I_{a_k,b_k}(u)}(x)\,F_X(\d x)\, \d u\\
&= \int_0^1\int_{\R} \one_{\bigcup_{k=1}^nI_{a_k,b_k}(u)}(x)\, F_X(\d x)\, \d u\\
&\leq \mnen{\frac{2\ep}3+\int_{\frac\ep3}^{1-\frac\ep3}\int_{\R} \one_{\bigcup_{k=1}^nI_{a_k,b_k}(u)}(x)\, F_X(\d x)\, \d u}\\
&= \mnen{\frac{2\ep}3+\int_{\frac\ep3}^{1-\frac\ep3}\int_{[c,d]} \one_{\bigcup_{k=1}^nI_{a_k,b_k}(u)}(x)\, F_X(\d x)\, \d u<\ep,}
\end{align*}
\mnen{where, in the second-to-last step, we used the fact that $I_{a_k,b_k}(u)\subseteq I_{a,b}(u)\subseteq[c,d]$ for all $k=1,\ldots, n$ and $u\in \big[\frac\ep3,1-\frac\ep3\big]$.}\  This shows that $F_{X+Y}$ is absolutely continuous, and the proof is complete.
\end{proof}

\begin{remark}\label{rem.density}
Consider the situation of Theorem \ref{thm.main}, and let $\pi$ be a stochastic kernel satisfying \eqref{eq.kernelrep} and \eqref{eq.def.prd}.\ For $x\in I$ and $B\in \cB$, let
\begin{equation}\label{eq.def.piZ}
  \pi_Z(x,B)
  :=\int_S \one_B\big(g(x,y)\big)\,\pi(x,\d y)
\end{equation}
be the regular conditional distribution of $Z=g(X,Y)$ given $X$. Defining the conditional quantile
\begin{equation}\label{eq.cond.quantile}
q(x,u):=\min\big\{a\in \R\colon \pi_Z\big(x,(-\infty,a]\big)\geq u\big\}\quad\text{for all }x\in \mnen{I}\text{ and }u\in (0,1),
\end{equation}
reversing the transformations in the proof of Theorem \ref{thm.main}, we find that
\[
q(x,u)=h(x)+f\big(h(x),u\big)\quad \text{for all }x\in I\text{ and }u\in (0,1),
\]
which implies that
\begin{equation}\label{eq.conditionhforq}
 q(x_2,u_2)-q(x_1,u_1)\geq h(x_2)-h(x_1)\quad 
\end{equation}
for all $x_1,x_2\in I$ with $x_1\leq x_2$ and $u_1,u_2\in (0,1)$ with $u_1\leq u_2$.\
Since $g(X,Y)\stackrel{\d}{=}q(X,U)$, we have thus obtained an alternative reduction to the case, where $S=(0,1)$, $\mathscr S$ is the Borel $\si$-algebra on $(0,1)$, $Y$ is uniform on $(0,1)$ and independent of $X$, and $g=q$.

Now, Equation \ref{eq.conditionhforq} implies that $q(\, \cdot\, , u)$ is strictly increasing for each $u\in (0,1)$.\ For $u\in (0,1)$, we define $\Phi_u\colon q(I,u)\to \R$ by
\[
\Phi_u\big(q(x,u)\big):=h(x)\quad \text{for all }x\in I.
\]
Then, Equation \ref{eq.conditionhforq} implies that
\[
0\leq \Phi_u(z_2)-\Phi_u(z_1)\leq z_2-z_1
\]
  for all $z_2,z_1\in q(I,u)$ with $z_1\leq z_2$ and $u\in (0,1)$, so that $\Phi_u$ is nondecreasing and $1$-Lipschitz for all $u\in (0,1)$.\ Now, we define $\overline q(I,u)$ and $q^{-1}(\,\cdot\, ,u)$ by \eqref{eq.hbar} and \eqref{eq.inverseh} with $h=q(\, \cdot\,,u)$ for all $u\in (0,1)$, respectively.\ Since $h^{-1}$ is locally absolutely continuous and
  \[
   h^{-1}\big(\Phi_u(z)\big)=q^{-1}(z,u)\quad \text{for all }z\in q(I,u) \text{ and }u\in (0,1),
  \]
  we thus find that $q^{-1}(\, \cdot\,,u)\colon \overline q(I,u)\to \R$ is locally absolutely continuous for all $u\in (0,1)$.\ Moreover, by definition of the left-continuous conditional quantile $q$,
 \begin{equation}\label{eq.alternativedefqinv}
  q^{-1}(z,u)=\sup\big\{x\in I\colon \pi_Z\big(x,(-\infty,z]\big)\geq u\big\}\quad \text{for all }z\in \overline q(I,u)\text{ and }u\in (0,1).
 \end{equation}
Observe that, with the convention $\sup\emptyset:=-\infty$, the right-hand side in \eqref{eq.alternativedefqinv} defines a Borel measurable function $\R\times (0,1)\to \R\cup\{\pm\infty\}$ that is locally absolutely continuous on $\overline q(I,u)$ and coincides with the inverse of the strictly increasing function $q(\, \cdot\,,u)$ on its image $q(I,u)$ for fixed $u\in (0,1)$. 
 \end{remark}

 We are now in the position to derive an explicit representation for the density of $g(X,Y)$.

 \begin{theorem}\label{cor.density}
 In the situation of Theorem \ref{thm.main}, the density of $g(X,Y)$ w.r.t.\ the Lebesgue measure is given by
 \[
 f_{g(X,Y)}(z):=\int_0^1 \one_{q(I,u)}(z) f_X\big(q^{-1}(z,u)\big)\partial_z q^{-1}(z ,u)\, \d u \quad\text{for a.e.\ }z\in \R,
 \]
  where $f_X$ is the density of $X$ w.r.t.\ the Lebesgue measure, $q$ is given by \eqref{eq.cond.quantile}, and 
 \[
 q^{-1}(z,u):=\sup\big\{x\in I\colon \pi_Z\big(x,(-\infty,z]\big)\geq u\big\}\in \R\cup\{\pm\infty\},\quad \text{for }z\in \R\text{ and }u\in (0,1),
 \]
 with $\pi_Z$ given by \eqref{eq.def.piZ} and a stochastic kernel $\pi$ satisfying \eqref{eq.kernelrep} and \eqref{eq.def.prd}.\ Moreover, for all $u\in (0,1)$, the restriction of $q^{-1}(\,\cdot\,,u)$ to $\overline q(I,u)$ is the locally absolutely continuous inverse of $q(\, \cdot\,,u)$. In particular, for all $z\in q(I,u)$ and $u\in (0,1)$, we have $$q^{-1}(z,u)\in I\quad \text{and} \quad q\big(q^{-1}(z,u),u\big)=z,$$
 and $\partial_zq^{-1}$ is the partial derivative of $q^{-1}$ w.r.t.\ the first coordinate, which exists Lebesgue a.e.\ on the Borel measurable set $\big\{(z,u)\in \R\times (0,1)\colon z\in q(I,u)\big\}$.
 \end{theorem}

 \begin{proof}
 Let $$d(z, u):=\limsup_{n\to \infty} n\Big(q^{-1}\big(z+\tfrac1n, u\big)-q^{-1}(z, u)\Big)\in \R\cup\{\pm\infty\}$$
 for all $(z,u)\in \R\times (0,1)$ with $q^{-1}(z,u)\in \R$.\ By Remark \ref{rem.density}, the map $q^{-1}(\,\cdot\,, u)$ is locally absolutely continuous on $\overline q(I,u)$ and the weak derivative equals $d(\, \cdot\,, u)$ Lebesgue a.e.\ on $\overline q(I,u)$.\ Defining
 \[
 L_u:=\bigg\{z\in \overline q(I,u)\colon \lim_{r\to 0}\frac{1}{2r}\int_{-r}^r |d(z+s,u)-d(z,u)|\, \d s=0\bigg\}\quad \text{for }u\in (0,1),
 \]
 by \cite[Theorem 7.7, p.\ 138]{zbMATH01022658}, the set $\overline q(I,u)\setminus L_u$ has Lebesgue measure zero 
 for all $u\in (0,1)$ and, by \cite[Theorem 7.11, p.\ 141]{zbMATH01022658},
 \[
  d(z,u)=\lim_{h\to 0} \frac{q^{-1}(z +h,u)-q^{-1}(z ,u)}h=:\partial_z q^{-1}(z ,u)\quad \text{for all\ }z\in L_u\text{ and }u\in (0,1).
 \]
 Hence, by \cite[Corollary 5.4.4]{bogachev-2007}, for all $a,b\in \R$ with $a\leq b$ and $u\in (0,1)$, we find that
  \begin{align*}
    \int_{q^{-1}(a,u)}^{q^{-1}(b,u)} \one_I(x) f_X(x)\, \d x&=\int_a^b \one_{\overline q(I,u)}(z) f_X\big(q^{-1}(z,u)\big) \partial_z q^{-1}(z ,u)\, \d z\\
    &=\int_a^b \one_{q(I,u)}(z) f_X\big(q^{-1}(z,u)\big) \partial_z q^{-1}(z,u)\, \d z,
  \end{align*}
  where, in the first step, we used the fact that $q^{-1}(z,u)\in I$ for all $z\in \overline q(I,u)$ and $u\in (0,1)$ and, in the second step, we used the fact that the function $q(\, \cdot\,,u)$ has at most countably many jumps for every $u\in (0,1)$, so that $\partial_z q^{-1}(\,\cdot\, ,u)$ vanishes on $\overline q(I,u)\setminus q(I,u)$ except for at most countably many points for fixed $u\in (0,1)$.\ Since $q$ and $q^{-1}$ are product measurable and the interval $I$ is Borel, the set
 \begin{align*}
 B&:=\big\{ (z,u)\in \R\times (0,1)\colon z\in q(I,u)\big\}\\
 &\;= \big\{ (z,u)\in \R\times (0,1)\colon q^{-1}(z,u)\in I\text{ and }q\big(q^{-1}(z,u),u\big)=z\big\}
 \end{align*}
 is Borel, which implies that the map $\R\times (0,1)\to \R,\,  (z,u)\mapsto \one_{q(I,u)}(z)$ is Borel measurable.\ Moreover, for all $r\in (0,1)$ and $n\in \N$ with $\frac{1}{n+1}\leq r< \frac{1}{n}$,
 \[
  \frac{1}{2r}\int_{-r}^r |d(z+s,u)-d(z,u)|\, \d s\leq \frac{n+1}{n}\cdot \frac{n}2 \int_{-\nicefrac1n}^{\nicefrac1n} |d(z+s,u)-d(z,u)|\, \d s
 \]
 and $d$ is product measurable, so that
  \begin{align*}
  L&:=\big\{ (z,u)\in B\colon z\in L_u\big\}= \bigg\{(z,u)\in B\colon \lim_{n\to \infty}\frac{n}{2}\int_{-\nicefrac1n}^{\nicefrac1n} |d(z+s,u)-d(z,u)|\, \d s=0\bigg\}
 \end{align*}
 is also Borel.\ 
 Hence, by Tonelli's theorem, the set $B\setminus L$ has Lebesgue measure zero since $q(I,u)\setminus L_u$ has Lebesgue measure zero for all $u\in (0,1)$, and
  \begin{align*}
  F_{g(X,Y)}(b)-F_{g(X,Y)}(a)&=\int_0^1 \int_I \one_{(a,b]}\big(q(x,u)\big)f_X(x)\, \d x\, \d u\\
  &=\int_0^1 \int_{q^{-1}(a,u)}^{q^{-1}(b,u)} \one_I(x)f_X(x)\, \d x\, \d u\\
  & =\int_0^1 \int_a^b \one_{q(I,u)}(z) f_X\big(q^{-1}(z,u)\big) \partial_z q^{-1}(z ,u)\, \d z\\
  &=\int_a^b f_{g(X,Y)}(z)\, \d z.
  \end{align*}
  The proof is complete.
\end{proof}

\begin{corollary}[PRD regularization]\label{cor.regularization}
Let $\ep>0$ and let $a\colon \R^d\to\R$ be nondecreasing and Borel measurable \mnen{with $d\in \N$}.\ If $F_X$ is absolutely continuous and $Y$ is PRD on $X$, then
\[
        \ep X + a(Y)
\]
has an absolutely continuous distribution function.\ In particular, $\ep X+Y_1+\cdots+Y_d$ has an absolutely continuous distribution function.
\end{corollary}

\begin{proof}
This is a direct consequence of Theorem \ref{thm.main} for the functions $g(x,y)=\ep x+a(y)$ and $h(x)=\ep x$ for $x\in \R$ and $y\in \R^d$.
\end{proof}

The previous corollary can be understood as a smoothing principle. A small absolutely continuous perturbation of the form $\ep X$ regularizes any aggregate $a(Y)$ of a random vector $Y$ with a nondecreasing Borel measurable aggregation function $a\colon \R^d\to \R$, even when $a(Y)$ is discrete or singular, as long as the conditional distribution of $Y$ moves upward in stochastic order with $X$, i.e., $Y$ is PRD on $X$.\

\mnen{Let $I=\R$, and define
\[
\pi_{a(Y)}(x,B):=\int_S\one_{B}\big(a(y)\big)\,\pi(x,\d y)
\]
for $x\in \R$ and $B\in \cB$.\ Then, the left-continuous conditional quantile is given by
\begin{equation}\label{eq.densiyprdreg}
q(x,u)=\ep x+\min\big\{a\in \R \colon \pi_{a(Y)}\big(x,(-\infty,a]\big)\geq u\big\}\quad \text{for }x\in \R\text{ and }u\in (0,1). 
\end{equation}
In this case, $\overline q(\R,u)=\R$ for all $u\in (0,1)$, so that $q^{-1}(\, \cdot\,,u)$ is locally absolutely continuous on $\R$ and $\partial_z q^{-1}(\, \cdot\,,u)$ vanishes Lebesgue-a.e.\ on $\R\setminus q(\R,u)$ for all $u\in (0,1)$.\ By Theorem \ref{cor.density}, the density of $\epsilon X+a(Y)$ is therefore given by
\[
f_{\epsilon X+a(Y)}(z)=\int_0^1 f_X\big(q^{-1}(z,u)\big)\partial_z q^{-1}(z,u)\, \d u\quad \text{for a.e.\ }z\in \R.
\]
In the case, where $X$ and $Y$ are independent, choosing $\pi(x,\,\cdot\,):=\P\circ Y^{-1}$ for all $x\in \R$, Equation \eqref{eq.densiyprdreg} yields
\[
q^{-1}(z,u)=\frac{z-q_{a(Y)}(u)}\ep\quad \text{for all }z\in \R\text{ and }u\in (0,1),
\]
where $q_{a(Y)}$ denotes the left-continuous quantile of $a(Y)$. Hence, in the independent case, the density of $\ep X+a(Y)$ is given by 
 \begin{align*}
    f_{\ep X+a(Y)}(z)  &=  \frac1\ep\int_0^1  f_X\bigg(\frac{z-q_{a(Y)}(u)}{\ep}\bigg)\, \d u\\
    &=\frac1\ep\int_{\R^d} f_X\bigg(\frac{z-a(y)}{\ep}\bigg)\, \P\circ Y^{-1}(\d y)\quad \text{for a.e.\ }z\in \R.
 \end{align*}}
  Corollary \ref{cor.regularization} \mnen{thus} extends this familiar smoothing mechanism beyond independence.\ We point out that, in a financial context, the existence of an independent random variable is often rather restrictive, whereas the existence of a systemic factor on which all other losses depend positively is rather typical.\ This is supported by empirical evidence, indicating that, in times of stress, losses tend to behave in a comonotonic way, cf.\ \cite{DasUppal2004, ContKokholm2013, ContWagalath2016}.

We illustrate the PRD regularization for the Cantor distribution in the comonotonic case.

\begin{example}
 Let $ S=\R$, $\mathscr S=\cB$, $\ep>0$, $X$ be uniform on $(0,1)$, and $Y$ be Cantor distributed.\ If $Y$ is PRD on $X$, then $\ep X+Y$ has a density. If $X$ and $Y$ are even comonotonic, it follows that
 \[
 F_{\ep X+Y}=\big(\ep\cdot {\rm id}+c^{-1}\big)^{-1}
 \]
 is absolutely continuous, where $c$ is the Cantor function and $c^{-1}$ denotes the, say, right-continuous quantile function of $c$, despite the fact that $\ep\cdot {\rm id}+c$ is \textit{not} absolutely continuous.
\end{example}

We conclude with another corollary for the conditional expectation. It is an immediate consequence of Theorem \ref{thm.main} and Theorem \ref{cor.density} together with the observation that, for every $\cB_I\otimes \mathscr S$-measurable function $g\colon \mnen{I}\times S\to \R$ satisfying \eqref{eq.aggregation.property} with $h\colon I\to \R$ and $\int_S|g(x,y)|\,\pi(x,\d y)<\infty$ for all $x\in I$,
\begin{align*}
\int_S g(x_2,y)\, \pi(x_2,\d y)-\int_S g(x_1,y)\, \pi(x_1,\d y)&\geq \int_S g(x_2,y)- g(x_1,y)\, \pi(x_1,\d y)\\
&\geq h(x_2)-h(x_1)
\end{align*}
for all $x_1,x_2\in I$ with $x_1\leq x_2$.

\begin{corollary}
 Let $g\colon \mnen{I}\times S\to \R$ be $\cB_I\otimes \mathscr S$-measurable with
\[
g(x_2,y_2)-g(x_1,y_1)\geq h(x_2)-h(x_1)
\]
for all $x_1,x_2\in \mnen{I}$ with $x_1\leq x_2$ and  $y_1,y_2\in S$ with $y_1\precsim y_2$, where $h\colon I\to \R$ is strictly increasing with locally absolutely continuous inverse.\ If $F_X$ is locally absolutely continuous, $Y$ is PRD on $X$, and $\pi$ is a stochastic kernel satisfying \eqref{eq.kernelrep}, \eqref{eq.def.prd}, and $\int_S|g(x,y)|\,\pi(x,\d y)<\infty$ for all $x\in I$, then
 \[
 \E[g(X,Y)|X]:=\int_S g(X,y)\, \pi(X,\d y)
 \]
 has a locally absolutely continuous distribution function.\ In this case, the density of the random variable $\E[g(X,Y)|X]$ w.r.t.\ the Lebesgue measure is given by
 \[
  f_{\E[g(X,Y)|X]}(z):= \one_{q(I)}(z)f_X\big(q^{-1}(z)\big) \frac{\d}{\d z}q^{-1}(z)\quad \text{for a.e.\ }z\in \R,
 \]
 where $q^{-1}\colon \overline q(I)\to I$ is the locally absolutely continuous inverse of 
 \[
 q\colon I\to \R,\quad x\mapsto \int_S g(x,y)\, \pi(x,\d y)
 \]
 and $\frac{\d}{\d z}q^{-1}$ is the derivative of $q^{-1}$, which exists Lebesgue a.e.\ on $\overline q(I)$.
\end{corollary}

%\bibliographystyle{abbrv}
%\bibliography{references}

\end{document}